\documentclass{amsart}

\usepackage{amssymb}
\usepackage{amsmath}
\usepackage{amsthm}
\usepackage{amscd,mdwlist}
\usepackage{hyperref}
\usepackage{xcolor}
\usepackage{cite}
\usepackage{bbm}
\usepackage{graphicx}
\usepackage{subcaption}
   \usepackage{etaremune}
    \theoremstyle{plain}
\newtheorem{theorem}{Theorem}[section]
\newtheorem{proposition}{Proposition}[section]
\newtheorem{corollary}{Corollary}[section]
\newtheorem{lemma}{Lemma}[section]

\theoremstyle{remark}
\newtheorem{remark}{Remark}[section]
\newtheorem{examples}{Examples}[section]

\numberwithin{equation}{section}

\DeclareMathOperator{\supp}{supp}

\DeclareMathOperator{\im}{Im}

\begin{document}

\title[Laplacians on partially and generalized hyperbolic attractors]{Self-adjoint Laplacians on partially and generalized hyperbolic attractors}

\author{Shayan Alikhanloo$^1$, Michael Hinz$^2$}
\thanks{$^1$, $^2$ Research supported by the DFG IRTG 2235: 'Searching for the regular in the irregular: Analysis of singular and random systems'.}
\address{$^1$Fakult\"at f\"ur Mathematik, Universit\"at Bielefeld, Postfach 100131, 33501 Bielefeld, Germany}
\email{salikhan@math.uni-bielefeld.de}
\address{$^2$Fakult\"at f\"ur Mathematik, Universit\"at Bielefeld, Postfach 100131, 33501 Bielefeld, Germany}
\email{mhinz@math.uni-bielefeld.de}

\begin{abstract}
We construct self-adjoint Laplacians and symmetric Markov semigroups on partially hyperbolic attractors and on hyperbolic attractors with singularities, endowed with Gibbs u-measures. If the measure has full support, we can also guarantee the existence of an associated symmetric Hunt diffusion process. In the special case of partially hyperbolic diffeomorphisms induced by geodesic flows on manifolds of negative sectional curvature the Laplacians we consider are self-adjoint extensions of well-known classical leafwise Laplacians. 
\tableofcontents
\end{abstract}

\keywords{Hyperbolic attractors, Gibbs u-measures, Dirichlet forms, self-adjoint operators, semigroups, diffusions}
\subjclass[2010]{31C25, 37D10, 37D30, 37D35, 37D40, 37D45, 47A07, 47B25, 47D07, 60J60}

\maketitle

\section{Introduction}

We construct self-adjoint Laplacians on partially hyperbolic attractors, \cite{BrinPesin, BDPP08, HaPe, Pesin, PesinSinai, FAR07}, \cite[Section 5]{CLP17}, and on hyperbolic attractors with singularities, \cite{KS86, L90, Pesin92, Sataev92, Schmeling98}, \cite[Section 8]{CLP17}, endowed with Gibbs u-measures, \cite{CLP17, Pesin92, PesinSinai}. In \cite{AH20} we had already studied self-adjoint Laplacians on uniformly hyperbolic attractors, endowed with SRB-measures, and the present article may be viewed as a continuation of this research. 

Here we have two principal goals: The first is to define such Laplacians in situations that are more general than uniform hyperbolicity, but in which the same method can still be applied rather easily. This allows to extend our analysis to many prominent classes of examples, such as geodesic flows on negatively curved manifolds, \cite{Anosov67, BarreiraPesin, BrinStuck, BG05, EinsiedlerWard, FH19, H17, KatokHasselblatt, Poll93}, or attractors of Lorenz, Lozi or Belykh type, \cite{Belykh, Levy, Lozi, Misiurewicz, Pesin92, Sataev99} or
\cite[Section 8]{CLP17}. The second principal goal is to point out that in the special case of geodesic flows our construction recovers an analysis that had already been established a long time ago, see for instance \cite{Yue, Yue95}. The extension of this kind of analysis to the attractors of dissipative hyperbolic dynamical systems, started in \cite{AH20} and continued here, is new. Since we wish to construct Laplacians self-adjoint with respect to Gibbs u-measures, they must locally be superpositions of symmetric Laplacians on local unstable manifolds, endowed with the conditional measures. The main difference to the geodesic flow case is that in general the conditional densities of the Gibbs u-measure may only be H\"older continuous, and this is not sufficient to introduce a 'classical' leafwise Laplacian on functions that are $C^2$ in the unstable directions. We view the situation from a quadratic forms perspective, \cite{AR90, BaGL, BH91, Da89, FOT94}, as it is common in mathematical physics, \cite[Section VIII.6]{RS80}, and partial differential equations, \cite[Chapter 6]{Evans}. Based on the stable manifold theorem and the notion of Gibbs u-measure, it is not difficult to construct 'natural' Dirichlet forms, and we regard the unique self-adjoint operators associated with such forms as 'natural' self-adjoint Laplacians on the attractor. The existence of a measure with suitable properties is a main ingredient, see \cite{ARHTT} for a related study.

In \cite{AH20} we had already used the same approach to construct self-adjoint Laplacians on uniformly hyperbolic attractors. Here we distill a simplified abstract version of the basic argument, and this simplification allows to easily apply it to partially hyperbolic attractors and to attractors with singularities. Because compared to \cite{AH20} subtle details need to be changed or added (such as the definition of rectangles or an additional approximation step), we provide a proof for the main argument, Lemma \ref{L:closable}. The fact that the present analysis may be seen as a generalization of known results for geodesic flows - which are examples of partially hyperbolic systems - had not been discussed in \cite{AH20}.

We proceed as follows. In Section \ref{S:Dirichlet} we introduce a general abstract setup and prove our main results, Theorem \ref{T:closable} and Corollary \ref{C:semigroup}, on the existence of self-adjoint Laplacians and symmetric Markov semigroups. 
Section \ref{S:regularity} discusses the case when the Gibbs u-measure has full support, in this case we can invoke the theory of regular Dirichlet forms, \cite{FOT94}, and guarantee the existence of an associated symmetric Hunt diffusion process, Theorem \ref{T:process}. In Section \ref{S:PH} we recall basic definitions on partially hyperbolic maps and attractors and well-known results on the existence of Gibbs u-measures. We then observe that Theorem \ref{T:closable}, Corollary \ref{C:semigroup} and, in some cases, also Theorem \ref{T:process} apply to  partially hyperbolic attractors and yield self-adjoint Laplacians, symmetric semigroups and, in some cases, diffusion processes, Corollary \ref{C:mainresults}. 
In Section \ref{S:geodesicflow} we put special emphasis on partially hyperbolic diffeomorphisms induced by geodesic flows, because, as mentioned, this allows to explain that the Laplacians constructed here generalize formerly known special cases, Remarks \ref{R:orientation1} and \ref{R:orientation2}. Section \ref{S:GH} contains a discussion of hyperbolic attractors with singularities and Gibbs u-measures on them and the observation that the results from Section \ref{S:Dirichlet} also apply to these cases, Corollary \ref{C:mainresultsGH}. This discussion of hyperbolic attractors with singularities motivates the notation we employ in Sections \ref{S:Dirichlet} and \ref{S:regularity}.

\section{Dirichlet forms and self-adjoint Laplacians}\label{S:Dirichlet}

Let $M$ be a smooth Riemannian manifold, $U\subseteq M$ a relatively compact open subset and $r\geq 1$. We assume that there is a sequence $D_1^-\subseteq D_2^-\subseteq \dots$ of compact subsets $D_\ell^-$, all contained in $\overline{U}$, and that for each point $z$ in the union
\begin{equation}\label{E:Dminus}
D^-=\bigcup_{\ell\geq 1} D_\ell^-
\end{equation}
there is an immersed submanifold $W(z)$ of $U$ of class $C^r$ that contains $z$. We assume further that for any $\ell\geq 1$ the set $D_\ell^-$ can be covered by finitely many subsets $\mathcal{R}_{\ell,i}\subseteq D_\ell^-$, $i=1,...,n_\ell$, each of which admits a partition $\mathcal{P}_{\ell,i}$ into open subsets $B$ of the submanifolds $W(z)$. As usual we refer to these sets $\mathcal{R}_{\ell,i}$ as \emph{rectangles}. We finally assume that $\mu$ is a Borel probability measure on $D^-$ such that  on each $\mathcal{R}_{\ell,i}$ having positive measure $\mu(\mathcal{R}_{\ell,i})>0$ the disintegration identity
\begin{equation}\label{E:disintegration}
\mu(E)=\int_{\mathcal{P}_{\ell,i}} \, \mu_{B}(E)\, \hat{\mu}_{\mathcal{P}_{\ell,i}}(dB), \quad \text{$E\subseteq \mathcal{R}_{\ell,i}$ Borel,}
\end{equation}
holds, where $\hat{\mu}_{\mathcal{P}_{\ell,i}}$ is the pushforward of $\mu$ under the canonical projection onto the elements $B$ of the partition $\mathcal{P}_{\ell,i}$ and for each $B\in \mathcal{P}_{\ell,i}$ the symbol $\mu_B$ denotes the conditional measure on $B$. See for instance \cite{Rokhlin} or \cite{Viana16}.

Now let each of the immersed submanifolds $W(z)$ be endowed with the Riemannian metric inherited from $M$ and let $m_{W(z)}$ denote the corresponding Riemannian volume on $W(z)$. For $\ell$, $i$ and $B\in \mathcal{P}_{\ell,i}$ let $m_{B}$ denote the restriction to $B$ of $m_{W(z)}$. We say that $\mu$ satisfies the\emph{ (AC)-property} if for any rectangle $\mathcal{R}_{\ell,i}$ of positive measure the conditional measures $\mu_{B}$ are absolutely continuous with respect to $m_{B}$ for $\hat{\mu}_{\mathcal{P}_{\ell,i}}$-a.e. $B$. We write $\mathcal{M}^\text{ac}_{bd}$ for the set of all Borel probability measures on $D^-$ with the (AC)-property and with Radon-Nikodym densities $d\mu_B/dm_B$ that are uniformly bounded and uniformly bounded away from zero. 

By $C(D^-)$ we denote the space of continuous functions on $D^-$. For $1\leq k\leq r$ we write $C^k(M)|_{D^-}$ for the space of retrictions to $D^-$ of functions from $C^k(M)$, clearly a dense subspace of $C(D^-)$ and $C^1(M)|_{D^-}$. We write $C^u(D^-)$ (resp. $C^{u,k}(D^-)$) for the space of Borel functions $\varphi:D^-\rightarrow \mathbb{R}$ whose restriction to any immersed submanifold $W(z)$, $z\in D^-$, is a continuous (resp. $C^k$-) function on $W(z)$. It is easily seen that $C(D^-)\subset C^{u}(D^-)$ and that $C^k(M)|_{D^-}\subset C^{u,k}(D^-)$ for any $1\leq k\leq r$, \cite[Theorem 5.27]{Lee}. The following facts are straightforward, see \cite[Propositions 4.2 and 5.1]{AH20} for proofs.

\begin{proposition}\label{P:dense}\mbox{}
\begin{enumerate}
\item[(i)] For any function $g\in C^1(M)$ and any $x\in D^-$ we have
\begin{equation*}\label{E:gradientineq}
\left\|\nabla_{W(x)} g|_{D^-}(x)\right\|_{T_xW(x)}\leq \left\|\nabla_M g(x)\right\|_{T_xM}.
\end{equation*}
\item[(ii)] For any finite Borel measure $\mu$ on $D^-$ and any $1\leq k\leq r$ the space $C^k(M)|_{D^-}$ is a dense subspace of $L^2(D^-,\mu)$.
\end{enumerate}
\end{proposition}

By a \emph{quadratic form} on $L^2(D^-,\mu)$ we mean a densely defined nonnegative definite symmetric bilinear form on $L^2(D^-,\mu)$, i.e. a pair $(\mathcal{E},\mathcal{D}(\mathcal{E}))$, where $\mathcal{D}(\mathcal{E})$ is a dense subspace of $L^2(D^-,\mu)$ and $\mathcal{E}$ is a nonnegative definite and symmetric bilinear form on $\mathcal{D}(\mathcal{E})$. We employ the notation $\mathcal{E}(\varphi):=\mathcal{E}(\varphi,\varphi)$. A quadratic form $(\mathcal{E},\mathcal{D}(\mathcal{E}))$ is said to be \emph{closed}, \cite[Section VIII.6]{RS80}, if $\mathcal{D}(\mathcal{E})$ is a Hilbert space with respect to the norm
\[\Vert \varphi\Vert_{\mathcal{D}(\mathcal{E})}:=\sqrt{\mathcal{E}(\varphi)+\Vert \varphi\Vert_{L^2(D^-,\mu)}^2}, \quad  \varphi\in \mathcal{D}(\mathcal{E}).\] 
A quadratic form $(\mathcal{E},\mathcal{D}(\mathcal{E}))$ is said to be \emph{closable} if it possesses a closed \emph{extension}, i.e. there is a closed form $(\mathcal{E}',\mathcal{D}(\mathcal{E}'))$ such that $\mathcal{D}(\mathcal{E})\subseteq \mathcal{D}(\mathcal{E}')$ and $\mathcal{E}=\mathcal{E}'$ on $\mathcal{D}(\mathcal{E})$. The smallest closed extension of a closable form is referred to as its \emph{closure}.

A closed quadratic form $(\mathcal{E},\mathcal{D}(\mathcal{E}))$ on $L^2(D^-,\mu)$ is called a \emph{Dirichlet form} if 
$\varphi\wedge 1 \in \mathcal{D}(\mathcal{E})$ for any $\varphi\in\mathcal{D}(\mathcal{E})$ and $\mathcal{E}(\varphi\wedge 1)\leq \mathcal{E}(\varphi)$, see \cite[Chapter I, 1.1.1 and 3.3.1]{BH91}, \cite[Chapter 1]{FOT94} or \cite{BaGL, Da89}. A Dirichlet form $(\mathcal{E},\mathcal{D}(\mathcal{E}))$ is called \emph{conservative} if $1\in\mathcal{D}(\mathcal{E})$ and $\mathcal{E}(1)=0$, \cite[p. 49]{FOT94}. A Dirichlet form $(\mathcal{E},\mathcal{D}(\mathcal{E}))$ with $1\in\mathcal{D}(\mathcal{E})$ is called \emph{local} if $\mathcal{E}(F(\varphi),G(\varphi))=0$ for any $F,G\in C_c^\infty(\mathbb{R})$ with disjoint supports and any $\varphi\in\mathcal{D}(\mathcal{E})$, \cite[Chapter I, Corollary 5.1.4]{BH91}.

Given $\varphi\in C^{u,1}(D^-)$ we write
\begin{equation}\label{E:gradient}
\nabla \varphi(z):=\nabla_{W(z)}(\varphi|_{W(z)})(z),\quad z\in D^-,
\end{equation}
to define the gradient operator $\nabla$ on $C^{u,1}(D^-)$. Let 
\begin{equation}\label{E:initialdomain}
\overline{\mathcal{D}}_0(\mathcal{E}^{(\mu)}):=\left\lbrace \varphi\in L^2(D^-,\mu)\cap C^{u,1}(D^-):\ \int_{D^-}\left\|\nabla \varphi(z)\right\|_{T_zW(z)}^2\:\mu(dz)<+\infty\right\rbrace
\end{equation}
and 
\begin{equation}\label{E:defquadform}
\mathcal{E}^{(\mu)}(\varphi):=\int_{D^-} \left\|\nabla\varphi (z)\right\|_{T_zW(z)}^2\:\mu(dz),\quad \varphi\in \overline{\mathcal{D}}_0(\mathcal{E}^{(\mu)}).
\end{equation}
We also use the notation 
\[\mathcal{D}_0(\mathcal{E}^{(\mu)}):=\overline{\mathcal{D}}_0(\mathcal{E}^{(\mu)})\cap C(D^-).\]
The following is our main existence result for self-adjoint Laplacians.

\begin{theorem}\label{T:closable}
Suppose that $\mu\in \mathcal{M}^\text{ac}_{bd}$.
\begin{enumerate}
\item[(i)] The quadratic form $(\mathcal{E}^{(\mu)},\overline{\mathcal{D}}_0(\mathcal{E}^{(\mu)}))$ on $L^2(D^-,\mu)$ is closable, and its closure $(\mathcal{E}^{(\mu)}, \overline{\mathcal{D}}(\mathcal{E}^{(\mu)}))$ is a local conservative Dirichlet form. 
\item[(ii)] There exists a unique non-positive definite self-adjoint operator $(\overline{\mathcal{L}}^{(\mu)},\mathcal{D}(\overline{\mathcal{L}}^{(\mu)}))$ on $L^2(D^-,\mu)$ such that 
\begin{equation*}\label{E:GaussGreen}
\int_{D^-} \overline{\mathcal{L}}^{(\mu)}u\:\varphi\:d\mu=-\mathcal{E}^{(\mu)}(u,\varphi)
\end{equation*}
for all $u\in \mathcal{D}(\overline{\mathcal{L}}^{(\mu)})$ and $\varphi\in \overline{\mathcal{D}}(\mathcal{E}^{(\mu)})$. In particular, we have 
\begin{equation}\label{E:harmonic}
\int_{D^-} \overline{\mathcal{L}}^{(\mu)}u\:d\mu=0,\quad u\in \mathcal{D}(\overline{\mathcal{L}}^{(\mu)}).
\end{equation}
\item[(iii)] Corresponding statements are true for the quadratic form $(\mathcal{E}^{(\mu)},\mathcal{D}_0(\mathcal{E}^{(\mu)}))$, its closure $(\mathcal{E}^{(\mu)}, \mathcal{D}(\mathcal{E}^{(\mu)}))$ and the generator $(\mathcal{L}^{(\mu)},\mathcal{D}(\mathcal{L}^{(\mu)}))$ of the latter. The Dirichlet form $(\mathcal{E}^{(\mu)}, \overline{\mathcal{D}}(\mathcal{E}^{(\mu)}))$ is an extension of $(\mathcal{E}^{(\mu)}, \mathcal{D}(\mathcal{E}^{(\mu)}))$.
\end{enumerate}
\end{theorem}
 
\begin{remark}\mbox{}
\begin{enumerate}
\item[(i)] The self-adjoint operators $(\overline{\mathcal{L}}^{(\mu)},\mathcal{D}(\overline{\mathcal{L}}^{(\mu)}))$ and $(\mathcal{L}^{(\mu)},\mathcal{D}(\mathcal{L}^{(\mu)}))$ may be seen as natural Laplacians on $D^-$. More precisely, they should be perceived as analogs of Laplacians on weighted manifolds in the sense of \cite[Definition 3.17]{Grigoryanbook}, see Section \ref{S:geodesicflow}.
\item[(ii)] One can run the argument with different choices of a priori domains (\ref{E:initialdomain}), and in general they lead to different closed forms and therefore also to different Laplacians. We concentrate on the closed form $(\mathcal{E}^{(\mu)},\overline{\mathcal{D}}(\mathcal{E}^{(\mu)}))$, because its domain is maximal from a 'transversal' point of view, and on the closed form $(\mathcal{E}^{(\mu)}, \mathcal{D}(\mathcal{E}^{(\mu)}))$, because its domain is well-connected to the topology of $D^-$. The different possible choices of a priori domains are connected to abstract Dirichlet problems, but this will be discussed elsewhere.
\end{enumerate}
\end{remark} 
 
Recall that a strongly continuous semigroup $(P_t)_{t>0}$ on $L^2(D^-,\mu)$ is said to be \emph{symmetric} if $\left\langle P_tu,v\right\rangle_{L^2(D^-,\mu)}= \left\langle u, P_tv\right\rangle_{L^2(D^-,\mu)}$
for every $u,v\in L^2(D^-,\mu)$ and all $t>0$, and \emph{(sub-) Markov} if for all $t>0$ and every $u\in L^2(D^-,\mu)$ such that $0\leq u\leq 1$ $\mu$-a.e. we have $0\leq P_tu\leq 1$ $\mu$-a.e. See \cite[Sections I.1 and I.2]{BH91} or \cite[Section 1.4]{FOT94}. A symmetric Markov semigroup $(P_t)_{t>0}$ on $L^2(D^-,\mu)$ is \emph{conservative} if $P_t 1=1$ for all $t>0$
and \emph{recurrent} if for every nonnegative $u\in L^1(D^-,\mu)$ we have $\int_0^\infty P_tu\:dt=\text{$0$ or $+\infty$}$ $\mu$-a.e. See \cite[p. 48/49]{FOT94}. The following statement is a consequence of \cite[Lemma 1.3.2 and Theorem 1.4.1]{FOT94} (see also \cite[Chapter I, Proposition 3.2.1]{BH91}) and \cite[Theorem 1.6.3 and Lemma 1.6.5]{FOT94}, its last claim is immediate from (\ref{E:harmonic}).

\begin{corollary}\label{C:semigroup}
Suppose that $\mu\in \mathcal{M}^\text{ac}_{bd}$. There exists a unique symmetric Markov semigroup $(\overline{P}_t)_{t>0}$ on $L^2(D^-,\mu)$ generated by $(\overline{\mathcal{L}}^{(\mu)},\mathcal{D}(\overline{\mathcal{L}}^{(\mu)}))$ as in Theorem \ref{T:closable} (ii). It is recurrent and conservative, and $\mu$ is an invariant measure for $(\overline{P}_t)_{t>0}$ in the sense that $\int_{D^-} \overline{P}_t u\: d\mu=\int_{D^-} u \:d\mu$ for all $u\in L^2(D^-,\mu)$. Corresponding statements are true for the unique symmetric Markov semigroup $(P_t)_{t>0}$ on $L^2(D^-,\mu)$ generated by $(\mathcal{L}^{(\mu)},\mathcal{D}(\mathcal{L}^{(\mu)}))$ as in Theorem \ref{T:closable} (iii). 
\end{corollary}

The key observation to prove Theorem \ref{T:closable} is the following.
\begin{lemma}\label{L:closable}
Suppose that $\mu\in \mathcal{M}^\text{ac}_{bd}$. Then $(\mathcal{E}^{(\mu)}, \overline{\mathcal{D}}_0(\mathcal{E}^{(\mu)}))$ is a closable quadratic form on $L^2(D^-,\mu)$.
\end{lemma}

With the aid of Proposition \ref{P:dense} and Lemma \ref{L:closable} Theorem \ref{T:closable} now follows easily using \cite[Theorem 3.1.1]{FOT94} and \cite[Theorem 1.3.1]{FOT94} or \cite[Chapter I, Propositions 1.2.2 or 3.2.1]{BH91}. See \cite[Theorem 5.1 and its proof]{AH20}.

We provide a proof of Lemma \ref{L:closable}. It is very similar to the proof of \cite[Lemma 5.1]{AH20}, but the definition of rectangles in the style of \cite{PesinSinai, Pesin92} used here allows a slight simplification, while the representation of $D^-$ as a union of the $D_\ell^-$ needs an additional approximation step.

\begin{proof}
Suppose that $(\varphi_j)_{j=1}^\infty\subset \overline{D}_0(\mathcal{E}^{(\mu)})$ is Cauchy w.r.t. the seminorm $(\mathcal{E}^{(\mu)})^{1/2}$ and such that $\lim_j \left\|\varphi_j\right\|_{L^2(D^-,\mu)}=0$.  Let $\varepsilon>0$. Choose $j_\varepsilon\geq 1$ large enough so that $\mathcal{E}^{(\mu)}(\varphi_j-\varphi_k)^{1/2}<\varepsilon/2$ for all $j,k\geq j_\varepsilon$ and choose $\ell\geq 1$ large enough such that 
\[\left(\int_{D^-\setminus D_\ell^-}\left\|\nabla \varphi_{j_\varepsilon}(z)\right\|_{T_zW(z)}^2\mu(dz)\right)^{1/2}<\frac{\varepsilon}{2}.\]
Then by the triangle inequality we have 
\begin{equation}\label{E:smallerthaneps}
\sup_{j\geq j_\varepsilon}\left(\int_{D^-\setminus D_\ell^-} \left\|\nabla \varphi_{j}(z)\right\|_{T_zW(z)}^2\mu(dz)\right)^{1/2}<\varepsilon.
\end{equation}
We claim that 
\begin{equation}\label{E:claim}
\lim_j \int_{D_\ell^-}\left\|\nabla \varphi_{j}(z)\right\|_{T_zW(z)}^2\mu(dz)=0.
\end{equation}
If so, then in combination with (\ref{E:smallerthaneps}) we obtain $\lim_j \mathcal{E}^{(\mu)}(\varphi_j)<\varepsilon$, and since $\varepsilon$ was arbitrary, this shows the closability of $(\mathcal{E}^{(\mu)}, \overline{\mathcal{D}}_0(\mathcal{E}^{(\mu)}))$. 

To verify (\ref{E:claim}) let $\mathcal{R}_{\ell,1},...,\mathcal{R}_{\ell,n_\ell}$ be a finite cover of $D_\ell^-$ by rectangles $\mathcal{R}_{\ell,i}$. We may assume they all have positive measure $\mu$; if not, we can simply omit those rectangles that have measure zero. For each $i$ the quantity
\begin{multline}\label{E:makesmall}
\int_{\mathcal{P}_{\ell,i}}\int_B \left\|\nabla_{W(\zeta)}(\varphi_j-\varphi_k)(\zeta)\right\|_{T_\zeta W(\zeta)}^2\mu_B(d\zeta)\hat{\mu}_{\mathcal{P}_{\ell,i}}(dB)\\
=\int_{\mathcal{R}_{\ell,i}}\left\|\nabla_{W(z)}(\varphi_j-\varphi_k)(z)\right\|_{T_z W(z)}^2\mu(dz)
\end{multline}
can be made arbitrarily small if $j$ and $k$ are chosen large enough. Here we have used (\ref{E:disintegration}). Clearly also 
\begin{equation}\label{E:gotozero}
\lim_j \int_{\mathcal{P}_{\ell,i}}\int_B (\varphi_j(\zeta))^2\mu_B(d\zeta)\hat{\mu}_{\mathcal{P}_{\ell,i}}(dB)
=\lim_j\int_{\mathcal{R}_{\ell,i}} (\varphi_j(z))^2\mu(dz)=0.
\end{equation}
Each $B\in \mathcal{P}_{\ell,i}$ is an open subset of some Riemannian manifold $W=W(z)$, hence itself a Riemannian manifold. Therefore the Dirichlet integral 
\begin{equation}\label{E:DirichletonB}
\psi\mapsto \int_B \left\|\nabla_W \psi(\zeta)\right\|_{T_\zeta W}^2\mu_B(d\zeta)
\end{equation}
on $B$ with domain 
\begin{equation}\label{E:Sobospace}
W^{1,2}(B)=\left\lbrace \psi\in L^2(B,\mu_B):\ \int_B \left\|\nabla_W \psi(\zeta)\right\|_{T_\zeta W}^2\mu_B(d\zeta)<+\infty\right\rbrace
\end{equation}
is a Dirichlet form on $L^2(B,\mu_B)$: The classical Dirichlet integral on $B$ with $m_B$ in place of $\mu_B$ in (\ref{E:DirichletonB}) and (\ref{E:Sobospace}) is well-known to be closed, see for instance \cite[Lemma 4.3]{Grigoryanbook}, and since the density $d\mu_B/dm_B$ is bounded and bounded away from zero, this easily carries over to (\ref{E:DirichletonB}) and (\ref{E:Sobospace}). It follows in particular that (\ref{E:DirichletonB}), endowed with the domain $C^1(B)$, is closable on $L^2(B,\mu_B)$. Therefore also the quadratic form 
\begin{multline}
\varphi\mapsto \int_{\mathcal{R}_{\ell,i}}\left\|\nabla_{W(z)} \varphi (z)\right\|_{T_z W(z)}^2\mu(dz)\notag\\
=\int_{\mathcal{P}_{\ell,i}}\int_B \left\|\nabla_{W(\zeta)}\varphi (\zeta)\right\|_{T_\zeta W(\zeta)}^2\mu_B(d\zeta)\hat{\mu}_{\mathcal{P}_{\ell,i}}(dB),
\end{multline}
endowed with the domain $\overline{\mathcal{D}}_0(\mathcal{E}^{(\mu)})$, is closable on $L^2(D_\ell^-, \mu)$: This follows using a well-known superposition argument from \cite[Theorem 1.2]{AR90} (see also \cite[Chapter V, Proposition 3.11]{BH91} or \cite[Section 3.1. (2$^\circ$)]{FOT94}), details may be found in \cite[Proposition D.1]{AH20}. Together with (\ref{E:makesmall}) and (\ref{E:gotozero}) and the triangle inequality this implies that
\begin{multline}
\lim_j \left(\int_{D_\ell^-} \left\|\nabla \varphi_j(z)\right\|_{ T_zW(z)}^2 \mu(dz)\right)^{1/2}\notag\\
\leq \sum_{i=1}^{n_\ell}\lim_j\left(\int_{\mathcal{R}_{\ell,i}}\left\|\nabla \varphi_j(z)\right\|_{ T_zW(z)}^2 \mu(dz)\right)^{1/2}=0,
\end{multline}
what shows (\ref{E:claim}).
\end{proof}

\section{Regularity and symmetric diffusion processes}\label{S:regularity}

If $D^-$ is closed (hence compact) and $\supp\mu=D^-$ we write $\Lambda:=D^-$. In this case we can employ further results from the theory of regular Dirichlet forms, \cite{FOT94}. A Dirichlet form $(\mathcal{E},\mathcal{D}(\mathcal{E}))$ on $L^2(\Lambda,\mu)$ is \emph{regular} if $\mathcal{D}(\mathcal{E})\cap C(\Lambda)$ is dense in $\mathcal{D}(\mathcal{E})$ with $\Vert\cdot\Vert_{\mathcal{D}(\mathcal{E})}$-norm and dense in $C(\Lambda)$ with the uniform norm. A regular Dirichlet form $(\mathcal{E},\mathcal{D}(\mathcal{E}))$ on $L^2(\Lambda,\mu)$ is said to be \emph{strongly local} if $\mathcal{E}(u,v)=0$ whenever $u,v\in \mathcal{D}(\mathcal{E})\cap C(\Lambda)$ are such that $v$ is constant on $\supp u$. See \cite[p.6]{FOT94}.

Statement (i) in the following result is immediate, statement (ii) is a consequence of \cite[Theorems 7.2.1 and 7.2.2]{FOT94}.

\begin{theorem}\label{T:process}
Suppose that $\Lambda:=D^-$ is closed, $\mu\in \mathcal{M}_{bd}^\text{ac}$ and $\supp\mu=\Lambda$. 
\begin{enumerate}
\item[(i)] The Dirichlet form $(\mathcal{E}^{(\mu)}, \mathcal{D}(\mathcal{E}^{(\mu)}))$ is regular and strongly local.
\item[(ii)] There is a $\mu$-symmetric Hunt diffusion process $((X_t)_{t\geq 0}, (\mathbb{P}^x)_{x\in \Lambda\setminus \mathcal{N}})$ on $\Lambda$, with starting points $x$ outside some properly exceptional set $\mathcal{N}$, such that for all bounded Borel functions $u$ on $\Lambda$, any $t>0$ and $\mu$-a.e. $x\in \Lambda$ we have $P_tu(x)=\mathbb{E}^x[u(X_t)]$.
\end{enumerate}
\end{theorem}

A strong Markov process is called a \emph{diffusion process} if its paths are continuous almost surely, \cite[Section 4.5]{FOT94}. It is said to be a \emph{Hunt process} if it satisfies certain specific regularity properties, see \cite[Section I.9]{BG68} or \cite[Appendix A.2]{FOT94}. A Borel set $\mathcal{N}\subset\Lambda$ is said to be \emph{properly exceptional} for a Hunt process $((X_t)_{t\geq 0}, (\mathbb{P}^x)_{x\in \Lambda\setminus \mathcal{N}})$ if it is a $\mu$-null set and $\mathbb{P}^x(X_t\in\mathcal{N}\ \text{for some $t\geq 0$})=0$, $x\in \Lambda\setminus \mathcal{N}$, see \cite[p. 134 and Theorem 4.1.1]{FOT94}. A Hunt process $((X_t)_{t\geq 0}, (\mathbb{P}^x)_{x\in \Lambda\setminus \mathcal{N}})$ is said to be \emph{$\mu$-symmetric} if 
\[\int_\Lambda \mathbb{E}^x[u(X_t)]v(x) \:\mu(dx)= \int_\Lambda u(x) \mathbb{E}^x[v(X_t)] \:\mu(dx)\]
for any $t>0$ and every bounded Borel functions $u, v$ on $\Lambda$, see \cite[Lemma 4.1.3]{FOT94}.

\begin{remark}\mbox{}
\begin{enumerate}
\item[(i)] The Hunt process $((X_t)_{t\geq 0}, (\mathbb{P}^x)_{x\in \Lambda\setminus \mathcal{N}})$ is unique up to a suitable type of equivalence, see \cite[Theorem 4.2.7]{FOT94}. It has infinite life time, \cite[Problem 4.5.1]{FOT94}. It may be regarded as a natural analog (in the leafwise sense) of Brownian motion.
\item[(ii)] For every bounded Borel $u$ and any $t>0$ the function $x\mapsto \mathbb{E}^x[u(X_t)]$ is an $(\mathcal{E}^{(\mu)},\mathcal{D}(\mathcal{E}^{(\mu)}))$-quasi-continuous version of $P_tu$, \cite[Theorem 4.2.3]{FOT94}.
\end{enumerate}
\end{remark}

\section{Partially hyperbolic attractors}\label{S:PH}

Let $M$ be a smooth compact Riemannian manifold. A \emph{topological attractor} for a $C^{r+\alpha}$-diffeomorphism $f:M\rightarrow M$ is a compact subset $\Lambda\subseteq M$ for which there is a neighborhood $U$ such that $\overline{f(U)}\subseteq U$ and
\begin{equation}
\Lambda=\underset{n\geq0}\bigcap\, f^n(U).
\end{equation}
The set $\Lambda$ is $f$-invariant, i.e. $f(\Lambda)=\Lambda$, and it is the largest subset of $U$ with this property. See for instance  \cite{BrinStuck} or \cite{KatokHasselblatt}.

A compact $f$-invariant subset $\Lambda\subseteq M$ is said to be \textit{partially hyperbolic} if for each $z\in \Lambda$ there exists a continuous $df$-invariant splitting of the tangent space $T_zM=E^s(z)\oplus E^c(z)\oplus E^u(z)$ and there are constants $c>0$, $0<\lambda_1\leq\mu_1<\lambda_2\leq\mu_2<\lambda_3\leq\mu_3$ with $\mu_1<1<\lambda_3$ such that for each $z\in \Lambda$ we have 
\begin{align}
 c^{-1}\lambda^n_1\left\| v\right\|_{T_{z}M}\leq\left\|d_zf^n v\right\|_{T_{f^n(z)}M}&\leq c\mu^n_1 \left\| v \right\|_{T_zM} \qquad \text{for $v\in E^s(z)$ and $n\geq 0$};\notag\\
 c^{-1}\lambda^n_2\left\| v\right\|_{T_{z}M}\leq\left\|d_zf^n v\right\|_{T_{f^n(z)}M}&\leq c\mu^n_2 \left\| v \right\|_{T_zM} \qquad \text{for $v\in E^c(z)$ and $n\geq 0$};\notag\\
 c^{-1}\lambda^n_3\left\| v\right\|_{T_{z}M}\leq\left\|d_zf^n v\right\|_{T_{f^n(z)}M}&\leq c\mu^n_3 \left\| v \right\|_{T_zM} \qquad \text{for $v\in E^u(z)$ and $n\geq 0$}.\notag
\end{align}
The subspaces $E^s(z)$, $E^u(z)$ and $E^c(z)$ are called \emph{stable}, \emph{unstable} and \emph{central} subspaces at $z$, respectively. The tangent vectors in $E^c(z)$ may be contracted or expanded, but not as sharply as vectors in $E^s(z)$ and $E^u(z)$; the central direction is `dominated' by the hyperbolic behaviour of the stable and unstable directions. See \cite[Section 2.1.4]{HaPe} or \cite[p. 13/14]{Pesin}. If $E^c(z)=\lbrace0\rbrace$ then $\Lambda$ is called a \emph{uniformly hyperbolic set}. 

If $M$ itself is a partially hyperbolic set then $f$ is called a \emph{partially hyperbolic diffeomorphism}. If, in addition, $E^c(z)=\lbrace0\rbrace$ then $f$ is called an \emph{Anosov diffeomorphism}. 

A topological attractor $\Lambda\subseteq M$ is called a \emph{partially hyperbolic attractor} if it is a partially hyperbolic set.

Now let $\Lambda$ be a partially hyperbolic attractor for a $C^{r+\alpha}$-diffeomorphism $f$ as above. Given a point $z$ in $\Lambda$, by the Stable Manifold Theorem \cite[Sections 4.2--4.5]{Pesin} we can construct $C^r$-embedded submanifolds $V^s(z)$ and $V^u(z)$ of $M$, called respectively \emph{local stable} and \emph{local unstable manifolds at $z$} such that
\begin{equation}\label{E:localmfds}
T_zV^{s}(z)=E^s(z)\quad \text{ and }\quad T_zV^{u}(z)=E^u(z).
\end{equation}

The \emph{global stable} and \emph{global unstable manifolds} $W^s(z)$ and $W^u(z)$ are defined by
\begin{equation}\label{E:localtoglobal}
W^{s}(z)=\underset{n\geq0}\bigcup f^{-n}(V^{s}(f^{n}(z))) \quad \text{ and }\quad 
W^{u}(z)=\underset{n\geq0}\bigcup f^{n}(V^{u}(f^{-n}(z))),
\end{equation}
that is, by iterating the local stable and unstable manifolds $V^s(z)$ and $V^u(z)$ backward resp. forward. These manifolds $W^s(z)$ and $W^u(z)$ are $C^r$-immersed submanifolds of $M$. In the special case that $\Lambda=M$, they are the leaves of the so-called stable and unstable foliations in the sense of \cite[Section 3.1]{HaPe}, see also \cite[Chapter 4]{Pesin}. In general, the central subspace $E^c(z)$ is not integrable, \cite{HPS1, HPS2}. 

\begin{remark}
In the context of partial hyperbolicity the manifolds $V^s(z)$ and $V^u(z)$ in (\ref{E:localmfds}) are also referred to as \emph{local strongly stable} and \emph{local strongly unstable} manifolds at $z$, respectively. The manifolds $W^s(z)$ and $W^u(z)$ in (\ref{E:localtoglobal}) are also called the \emph{global strongly stable} and \emph{global strongly unstable} manifolds at $z$, respectively.
\end{remark}

Any partially hyperbolic attractor $\Lambda$ contains the global unstable manifolds of its points, 
\begin{equation}\label{E:partiallyglobal}
\Lambda=\bigcup_{z\in\Lambda} W^u(z), 
\end{equation}
see for instance \cite[Theorem 9.1]{HaPe}, and the union (\ref{E:partiallyglobal}) is disjoint.

As usual $B(x,\delta)$ denotes the open ball in $M$ with center $x\in M$ and radius $\delta>0$. Given $x\in \Lambda$ and $\delta>0$ let 
\begin{equation}\label{E:rectanglePH}
\mathcal{R}(x,\delta)=\bigcup_{z\in B(x,\delta)\cap \Lambda}V^u(z).
\end{equation}
For sufficiently small $\delta$ the set $\mathcal{R}(x,\delta)$ admits a measurable partition $\mathcal{P}_{\mathcal{R}(x,\delta)}$ into local unstable manifolds $V=V^u(z)$.

Recall that two measures on the same space are said to be \emph{equivalent} if they are mutually absolutely continuous. An $f$-invariant Borel probability measure $\mu$ on $\Lambda$ is called a \textit{Gibbs u-measure} if for any $x\in\Lambda$ and (sufficiently small) $\delta>0$ such that $\mu(\mathcal{R}(x,\delta))>0$ we have 
\begin{equation}\label{E:disintegrationPH}
\mu(E)=\int_{\mathcal{P}_{\mathcal{R}(x,\delta)}} \, \mu_{V}(E)\, \hat{\mu}_{\mathcal{P}_{\mathcal{R}(x,\delta)}}(dV), \quad \text{$E\subseteq \mathcal{R}(x,\delta)\ $ Borel,}
\end{equation}
where $\hat{\mu}_{\mathcal{P}_{\mathcal{R}(x,\delta)}}$ is the pushforward of $\mu$ under the canonical projection onto the elements $V$ of the partition $\mathcal{P}_{\mathcal{R}(x,\delta)}$, and the conditional measures $\mu_V$ are equivalent to the Riemannian volumes $m_V$ on the manifolds $V$. See for instance \cite[Section 5.2]{CLP17}.

Gibbs u-measures on partially hyperbolic attractors can be constructed in the same way as SRB-measures on uniformly hyperbolic attractors, \cite{CLP17}: Given a local unstable manifold $V^u(z)$, $z\in \Lambda$, one can consider the sequence $(\mu_n)_n$ of probability measures $\mu_n$ defined by
\begin{equation}\label{G:measure}
\mu_n:=\frac{1}{n}\sum_{i=0}^{n-1} f_*^im_{V^u(z)},
\end{equation}
where $f_*^im_{V^u(z)}$ is the pushforward of the measure $m_{V^u(z)}$. We quote the following result from \cite[Theorem 4]{PesinSinai}; further details and descriptions can be found in \cite[Section 5]{CLP17} and \cite[Section 9]{HaPe}. For statement (ii) see \cite[Theorem 5.4]{CLP17} or \cite{BDPP08}.

\begin{theorem}\label{T:Gibbsexist}
Assume that $f:M\to M$ is a $C^{1+\alpha}$ diffeomorphism and $\Lambda\subseteq M$ is a partially hyperbolic attractor. 
\begin{enumerate}
\item[(i)] There is a Gibbs u-measure $\mu$ on $\Lambda$ with uniformly bounded and H\"older continuous conditional densities $d\mu_V/dm_V$, that is, $\mu\in \mathcal{M}^\text{ac}_{bd}$. Any weak limit of $(\mu_n)_n$ as in (\ref{G:measure}) has these properties. 
\item[(ii)] If for every $z\in\Lambda$ the orbit of the global (strongly) unstable manifold $W^u(z)$ is
dense in $\Lambda$, then every Gibbs u-measure $\mu$ has support $\supp\mu=\Lambda$.
\end{enumerate}
\end{theorem}

Theorem \ref{T:Gibbsexist} makes Theorem \ref{T:closable} and its consequences applicable.

\begin{corollary}\label{C:mainresults}
Let $f$ and $\Lambda$ be as in Theorem \ref{T:Gibbsexist} and let $\mu\in \mathcal{M}^\text{ac}_{bd}$ be a Gibbs u-measure on $\Lambda$ with uniformly bounded densities.
\begin{enumerate}
\item[(i)] The quadratic forms $(\mathcal{E}^{(\mu)},\overline{\mathcal{D}}(\mathcal{E}^{(\mu)}))$ and $(\mathcal{E}^{(\mu)},\mathcal{D}(\mathcal{E}^{(\mu)}))$ as in Theorem \ref{T:closable} are local Dirichlet forms on $L^2(\Lambda,\mu)$,
their generators are self-adjoint operators on $L^2(\Lambda,\mu)$, and their semigroups are symmetric.
\item[(ii)] If for every $z\in\Lambda$ the orbit of the global (strongly) unstable manifold $W^u(z)$ is
dense in $\Lambda$, then $(\mathcal{E}^{(\mu)},\mathcal{D}(\mathcal{E}^{(\mu)}))$ is regular and strongly local, and there is an associated $\mu$-symmetric Hunt diffusion process as in Theorem \ref{T:process}.
\end{enumerate}
\end{corollary}

\begin{proof}
Setting $D_\ell^-:=\Lambda$, $\ell\geq 1$, we have $D^-=\Lambda$ in (\ref{E:Dminus}). From (\ref{E:rectanglePH}) and the compactness of $\Lambda$ it follows that $\Lambda$ admits a cover by finitely many rectangles $\mathcal{R}(x,\delta)$, each of which is partitioned into local unstable manifolds $V=V(z)$ that are open subsets of the global unstable manifolds $W^u(z)$. Since $\mu\in \mathcal{M}_{bd}^{ac}$ satisfies (\ref{E:disintegrationPH}), we see that (\ref{E:disintegration}) holds. 
\end{proof}

We briefly recall well-known examples for partially hyperbolic diffeomorphisms.

\begin{examples}
One class of examples of partially hyperbolic attractors is generated by \emph{direct products} $f\times g:M\times N\rightarrow M\times N$, where $f:M\rightarrow M$ is a partially hyperbolic diffeomorphism and $g:N\rightarrow N$ is a diffeomorphism whose dynamical behaviour is less sharp than that of $f$ in the sense of \cite[Section 2.3, Example 3]{Pesin} or \cite[Section 2.7]{FAR07}. Then the direct product $F=f\times g:M\times N\to M\times N$ defined by $F(x,y):=(f(x),g(y))$ is a partially hyperbolic diffeomorphism. A particularly simple case arises if $f$ is an Anosov diffeomorphism and $g$ is the identity.
\end{examples}

\begin{examples}
Suppose $f:M\rightarrow M$ is an Anosov diffeomorphism, $G$ is a (compact) Lie group $G$ and $\theta:M\rightarrow G$ is a smooth function. The \emph{skew product} $F_\theta:M\times G\rightarrow M\times G$ is defined by $F_\theta(x,y):=(f(x),\theta(x)y)$, $x\in M$, $y\in G$, it is a partially hyperbolic diffeomorphism. See for instance \cite[Section 2.3, Examples 4 and 5]{Pesin} or \cite[Section 2.9]{FAR07}.
\end{examples}

\begin{examples}
Let $\phi:\mathbb{R}\times M\rightarrow M$ be a flow (generated by a given vector field). By \textit{time-$t$ map} we mean the diffeomorphism $\phi(t,\cdot) : M \rightarrow M$. The flow $\phi$ is said to be \emph{partially hyperbolic} if its time-$1$ map $\phi(1,\cdot)$ is a partially hyperbolic diffeomorphism. A \emph{uniformly hyperbolic} (or \emph{Anosov}) \emph{flow} is a partially hyperbolic flow with one-dimensional central subspace $E^c(x)=\text{span} \lbrace{\frac{\partial}{\partial t}\vert}_{t=0}\, \phi(t,x)\rbrace$, $x\in M$. See for instance \cite[Definition 17.4.2]{KatokHasselblatt}. From any Anosov diffeomorphism of a compact Riemannian manifold $M$ one can construct Anosov flows called \emph{suspension flows}, see \cite[Section 1.11]{BrinStuck}, \cite[Section 0.3]{KatokHasselblatt} or \cite[p. 8]{Pesin}. If $\phi$ is a $C^{r+\alpha}$ Anosov flow, then the local (strongly) stable and unstable manifolds (\ref{E:localmfds}) can also be obtained from a continuous time-version of the Stable Manifold Theorem for flows \cite[Theorem 17.4.3]{KatokHasselblatt}, \cite[Section 7.3.5]{BarreiraPesin}, and also in (\ref{E:localtoglobal}) a positive real index can be used in place of $n$ and $\phi(t,\cdot)$ can replace $f^n$.
\end{examples}

A particular class of examples of Anosov flows is formed by geodesic flows on manifolds of negative sectional curvature. Since for these examples there is an established leafwise analysis to which the theory in Sections \ref{S:Dirichlet} and \ref{S:regularity} can be compared, we discuss them in a slightly more detailed manner in the next section.

\section{Geodesic flows on manifolds with negative curvature}\label{S:geodesicflow}
 
Let $M$ be a compact $C^r$ manifold endowed with a $C^r$ Riemannian metric, $r\geq2$. Given $x\in M$ and $v\in T_xM$, there is a unique geodesic $\gamma_{x,v}$ such that $\gamma_{x,v}(0)=x$ and $\dot{\gamma}_{x,v}(0)=v$. By the \emph{geodesic flow} on $M$ we mean the flow $g:\mathbb{R}\times TM\rightarrow TM$ on the tangent bundle $TM$, defined by
\begin{equation}\label{E:geoflow}
g(t,(x,v)):=(\gamma_{x,v}(t),\dot{\gamma}_{x,v}(t)).
\end{equation}
Since the length of the tangent vectors are preserved by the geodesic flow, i.e. $\Vert \dot{\gamma}_{x,v}(t)\Vert_{T_{\gamma_{x,v}(t)}M}=\Vert v \Vert_{T_xM}$, it is usual to consider the restriction of $g$ to the unit tangent bundle
\[T^1M=\lbrace (x,w)\in TM:\, \Vert w\Vert_{T_xM}=1\rbrace.\]
We assume that $M$ has negative sectional curvature. Then the $C^{r-1}$ geodesic flow $g:\mathbb{R}\times T^1M\rightarrow T^1M$ is an Anosov flow, \cite{Anosov67}, see for instance \cite[Theorem 9.4.1]{BG05} and \cite[Sections 17.5 and 17.6]{KatokHasselblatt}.

\begin{examples}
As a guiding 'example' consider the Poincar\'e half plane, that is, $\mathbb{H}=\{z\in\mathbb{C}:\ \im z>0\}$, endowed with the Riemannian metric 
\[\left\langle v,w\right\rangle_{T_z\mathbb{H}}=(\im z)^{-2}\left\langle v,w\right\rangle,\quad v,w\in T_z\mathbb{H}=\mathbb{C},\] where $\left\langle\cdot,\cdot\right\rangle$ is the usual inner product on $\mathbb{C}\cong \mathbb{R}^2$. The manifold $\mathbb{H}$ itself is not compact, so it does not fully fit our assumptions, but it is the universal cover of any closed hyperbolic surface. It is compactified by adjoining the ideal boundary $\mathbb{H}(\infty):=\{z\in\mathbb{C}: \im z=0\}$. 
Its unit tangent bunde $T^1\mathbb{H}=\{(z,v)\in \mathbb{H}\times\mathbb{C}: \left\|v\right\|_{T_z\mathbb{H}}=1\}$ is isomorphic to $PSL_2(\mathbb{R})$, and there is a related simple explicit description of the geodesic flow (\ref{E:geoflow}), see \cite[Section 1.2]{BarreiraPesin} or \cite[Sections 9.1 and 9.2]{EinsiedlerWard}. Geodesics are vertical lines or half-circles centered at points on $\mathbb{H}(\infty)$. Given $(z,v)\in T^1\mathbb{H}$, let $\gamma_{z,v}^-:=\lim_{t\to -\infty} \gamma_{z,v}(t)$, where $\gamma_{z,v}$ is the geodesic through $z$ in direction $v$. The circle $c(z,v)$ through $z$ and tangent to $\mathbb{H}(\infty)$ at $\gamma_{z,v}^-$ is called the \emph{horocycle} through $z$, and the global strong unstable manifold $W^{u}((z,v))$ through $(z,v)\in T^1\mathbb{H}$ is formed by all $(z',v')\in T^1\mathbb{H}$ such that $z'\in c(z,v)$ and $v'$ is the outer normal on $c(z,v)$ at $z'$. Further details may also be found in \cite[p. 119]{FH19} or \cite[Chapter 3]{H17}.
\end{examples}

Since the time-$1$ map $g(1,\cdot):T^1M\to T^1M$ of the geodesic flow (\ref{E:geoflow}) is a partially hyperbolic diffeomorphism, Theorem \ref{T:Gibbsexist} ensures the existence of a Gibbs u-measure $\mu\in \mathcal{M}^\text{ac}_{bd}$ on $T^1M$. On the other hand the \emph{Liouville measure} $\nu$ on $T^1M$, defined by
\begin{equation}\label{E:Liouville}
\int_{T^1M}h\,d\nu=\int_M\int_{T^1_xM}h(x,v)\,dm_{\mathbb{S}^{n-1}}\,dm_M(x),\quad h\in C(T^1M),
\end{equation}
where $m_M$ and $m_{\mathbb{S}^{n-1}}$ denote the Riemannian volume on $M$ and on the sphere $\mathbb{S}^{n-1}\equiv T^1_xM$, respectively, is invariant under $g(1,\cdot)$, see for instance \cite[Appendix C]{Poll93}. But this implies that $\mu=\nu$, \cite[Theorem 7.4.14]{FH19}, and using (\ref{E:Liouville}) it follows that $\supp\mu=T^1M$. It is well known that in the special case of a compact surface of constant negative curvature, this measure also coincides with the Bowen-Margulis measure, i.e. the measure of maximal entropy of the flow, see \cite[Chapter 11, Example 2]{PP90} or \cite{L90}.

By Corollary \ref{C:mainresults} the forms $(\mathcal{E}^{(\mu)},\overline{\mathcal{D}}(\mathcal{E}^{(\mu)}))$ and $(\mathcal{E}^{(\mu)},\mathcal{D}(\mathcal{E}^{(\mu)}))$ are local respectively strongly local and regular Dirichlet forms on $L^2(T^1M,\mu)$, and in Section \ref{S:Dirichlet} we suggested to regard their generators as generalized Laplacians. On the other hand, there is an established analysis involving leafwise Laplacians in the unstable directions, see for instance \cite{Yue, Yue95}. We briefly compare $(\overline{\mathcal{L}}^{(\mu)},\mathcal{D}(\overline{\mathcal{L}}^{(\mu)}))$ to the Laplacians studied in \cite{Yue, Yue95}.

For any rectangle $\mathcal{R}(x,\delta)$ and any local unstable manifold $V$ as it appears in the partition (\ref{E:rectanglePH}) of this rectangle, let $\varrho_V=d\mu_V/dm_V$ be the Radon-Nikodym density of the conditional measures $\mu_V$ of $\mu$ with respect to $m_V$ on $V$. Assume that $M$ and $f$ are of class $C^\infty$. Then each $\varrho_V$ is a $C^\infty(V)$-function. This $C^\infty$-regularity was proved in \cite[Lemma 2.1]{Yue}, which itself was based on \cite[Lemma 2.5]{LMM86}. Setting $\varrho(z):=\varrho_V(z)$ if $z$ is a point in the partition element $V$, we can define $\varrho$ as a $C^{\infty,u}$-function on $\mathcal{R}(x,\delta)$. On the other hand, a 'classical' leafwise Laplacian in the unstable directions can be defined by setting
\[\Delta \varphi(z):=\Delta_{W^u(z)}(\varphi|_{W^u(z)})(z),\quad z\in M,\]
for any $\varphi\in C^{2,u}(M)$. Here $\Delta_{W^u(z)}$ denotes the classical Laplace-Beltrami operator on the Riemannian manifold $W^u(z)$, defined on $C^2$-functions. Now recall that $\nabla$ defines the leafwise gradient on $C^{1,u}(M)$-functions as defined in (\ref{E:gradient}) (with $D^-=M$ and $W(z)=W^u(z)$). For functions $\varphi\in C^{2,u}(M)$ we can define 
\begin{equation}\label{E:Laplacian}
\Delta_\mu\varphi:=\Delta\varphi+\varrho^{-1}\langle\nabla \varrho,\nabla \varphi\rangle,\quad z\in \mathcal{R}(x,\delta),
\end{equation}
where we write the shortcut $\langle\nabla \varrho,\nabla \varphi\rangle$ for the function $z\mapsto\langle\nabla \varrho (z),\nabla \varphi (z)\rangle_{T_zW^u(z)}$. Using a smooth partition of unity, definition (\ref{E:Laplacian}) can meaningfully be extended to hold for all $z\in M$. The Laplacian $\Delta_\mu\varphi$ may be seen as a leafwise analog of the Laplacian on weighted manifolds, \cite[Definition 3.17]{Grigoryanbook}. It had already been studied in \cite{Yue}, see for instance \cite[Theorem 1']{Yue}, where $\mu$ had been shown to be an invariant measure for $\Delta_\mu$. This is fully consistent with our results in the sense that 
\begin{equation}\label{E:saext}
\text{$(\overline{\mathcal{L}}^{(\mu)},\mathcal{D}(\overline{\mathcal{L}}^{(\mu)}))$ is a self-adjoint extension of $(\Delta_\mu, C^{2,u}(M))$ on $L^2(T^1M,\mu)$, }
\end{equation}
note that we had independently observed the invariance of $\mu$ in (\ref{E:harmonic}).

\begin{remark}\label{R:orientation1}
Observation (\ref{E:saext}) means that for geodesic flows of class $C^\infty$ the theory in this article is simply an $L^2$-version of the smooth theory for the operator $\Delta_\mu$ as considered in \cite{Yue, Yue95}. A similar $L^2$-theory for the stable directions, including self-adjoint Laplacians, is studied extensively in \cite{Hamenstaedt97}.
\end{remark}

\begin{remark}\label{R:orientation2}
We chose the specific example of $C^\infty$-geodesic flows because it is widely known and because for this situation leafwise Laplacians of form (\ref{E:Laplacian}) had been studied in \cite{Yue}. A definition of leafwise Laplacians as in (\ref{E:Laplacian}) is possible whenever the unstable manifolds are $C^2$ and the conditional densities $C^1$, and this can be guaranteed also for certain more general classes of diffeomorphisms $f:M\to M$, see for instance \cite[Remark on p. 534]{LedrappierYoung1} or \cite[p. 168]{Yue95}.

For general partially hyperbolic attractors $\Lambda\subseteq M$ a definition of Laplacians $\Delta_\mu$ as in (\ref{E:Laplacian}) seems out of reach: The only regularity information for the conditional densities $\varrho_V$ is their H\"older continuity, too little to give a meaning to (\ref{E:Laplacian}). However, Theorem \ref{T:closable} ensures the existence of self-adjoint Laplacians as generators of quadratic forms, and their definition (\ref{E:defquadform}) and closedness posed no problem. The situation is very similar to the theory of weak solutions in partial differential equations: Divergence form elliptic second order differential operators with bounded measurable coefficients cannot be defined directly as classical operators on a space of $C^2$-functions, but 
are easily defined as the generators of corresponding quadratic forms. See for instance \cite[Chapter 6]{Evans}.
\end{remark}

\section{Hyperbolic attractors with singularities}\label{S:GH}

We consider a more general class of hyperbolic attractors induced by maps with discontinuities, it had been studied in \cite{Pesin92}. A short exposition may also be found in \cite[Section 8]{CLP17}. The notation in this section follows \cite{Pesin92}, up to minor details.

Let $M$ be a smooth Riemannian manifold. Let $U\subseteq M$ be a relatively compact open set and $N\subset U$ a closed subset. Let $f:U\setminus N\rightarrow U$ be a $C^{r+\alpha}$-diffeomorphism onto its image. We define 
\[N^+:=N\cup \partial U\]
and 
\[N^-:=\lbrace y\in U: \text{there are $z\in N^+$ and $z_n\in U\setminus N^+$ with $z_n\to z$ and $f(z_n)\to y$}\rbrace\]
and assume that $f$ is such that 
\begin{align}
\left\|d_z^2f \right\|&\leq c_1\,d(z, N^+)^{-\alpha_1} \qquad \text{for any $z\in U\setminus N$},\notag\\
\left\| d_z^2f^{-1}\right\|&\leq c_2\,d(z, N^-)^{-\alpha_2} \qquad \text{for any $z\in f(U\setminus N)$,}\notag
\end{align}
where $c_i>0$, $\alpha_i\geq0$, $i=1,2$, and $\left\|\cdot\right\|$ denotes the operator norm. A \emph{topological attractor with singularities} for $f$ is defined to be the set $\Lambda:=\overline{D}$ where
\begin{equation}\label{attractor}
D:=\bigcap_{n\geq0} f^n(U^+)\quad \text{and}\quad U^+:=\lbrace x\in U: f^n(x)\notin N^+,  n=0, 1, 2,\dots\rbrace.
\end{equation}

Given $z\in M$, $a>0$ and a subspace $P(z)\subseteq T_zM$, the \textit{cone} at $z$ around $P(z)$ with angle $\theta$ is the set $C(z, P(z), \theta):=\lbrace v\in T_zM: \measuredangle(v, P(z))\leq \theta\rbrace$. Here we write $\measuredangle(v, P):=\min_{w\in P} \measuredangle(v,w)$ for any $P\subseteq T_zM$, and we define $\measuredangle(P', P)$ for $P,P'\subseteq T_zM$ in a similar manner. 

A topological attractor with singularities $\Lambda$ is said to be a \textit{uniformly hyperbolic attractor with singularities} (or \emph{generalized hyperbolic attractor}) if there exist constants $c>0$, $\lambda \in (0,1)$, and $\theta(z)>0$ , $z\in U\setminus N^+$, together with subspaces $P^s(z),P^u(z)\subseteq T_zM$, $z\in U\setminus N^+$, of complementary dimension, such that the two families of \emph{stable} and \emph{unstable cones} 
\[C^s(z)=C^s(z, P^s(z), \theta(z))\quad \text{and}\quad C^u(z)=C^u(z, P^u(z), \theta(z))\] 
satisfy the following conditions:
\begin{enumerate}
\item[(i)] the angles $\measuredangle(C^s(z),C^u(z))$, $z\in U\setminus N^+$, are uniformly bounded away from zero, 
\item[(ii)] we have $df(C^u(z))\subseteq C^u(f(z))$ for any $z\in U\setminus N^+$ and $df^{-1}(C^s(z))\subseteq C^s(f^{-1}(z))$
for any $z\in f(U\setminus N^+)$, 
\item[(iii)] for any $n\geq 1$ we have
\begin{align}
\left\|d_zf^n v\right\|_{T_{f^n(z)}M}&\geq  c\lambda^{-n} \left\| v \right\|_{T_zM} \qquad \text{for $v\in C^u(z)$ and $z\in U^+$},\notag\\
\left\| d_zf^{-n}v\right\|_{T_{f^{-n}(z)}M}&\geq c\lambda^{-n} \left\| v \right\|_{T_zM} \qquad \text{for $v\in C^s(z)$ and $z\in f^n(U^+)$}.\notag
\end{align}
\end{enumerate}
See \cite[Section 1.3]{Pesin92} or \cite[Section 8]{CLP17}. 

In the following we assume that $\Lambda$ is a uniformly hyperbolic attractor with singularities, and we continue to use  the above notation. For any $z\in D$ the subspaces 
\[E^s(z)=\bigcap_{n\geq0} df^{-n} C^s(f^n(z)) \quad \text{and} \quad E^u(z)=\bigcap_{n\geq0} df^{n} C^u(f^{-n}(z))\]
form a splitting of the tangent space $T_zM=E^s(z)\oplus E^u(z)$ such that for any $n\geq0$
\begin{align}
\left\|d_zf^n v\right\|_{T_{f^n(z)}M}&\leq c\lambda^n \left\| v \right\|_{T_zM} \quad \quad\text{for $v\in E^s(z)$};\notag\\
\left\| d_zf^{-n}v\right\|_{T_{f^{-n}(z)}M}&\leq c\lambda^{n} \left\| v \right\|_{T_zM} \qquad \text{for $v\in E^u(z)$,}\notag
\end{align}
meaning that $D$ is a uniformly hyperbolic set containd in $\Lambda$, see \cite[p. 128]{Pesin92} or \cite{Pesin77}. An adapted version of the Stable Manifold Theorem, \cite[Proposition 4]{Pesin92}, guarantees that for sufficiently small $\varepsilon>0$
local stable manifolds $V^s(z)$, $z\in D^+_\varepsilon$, and local unstable manifolds $V^u(z)$, $z\in D^-_\varepsilon$, exist,  where for any $\ell\geq 1$ we write
\begin{align*}
D^+_{\varepsilon, \ell}&:=\lbrace z\in \Lambda: d(f^n(z), N^+)\geq \ell^{-1} e^{-\varepsilon n},\, n\geq 0\rbrace,\\
D^-_{\varepsilon, \ell}&:=\lbrace z\in \Lambda: d(f^{-n}(z), N^-)\geq \ell^{-1} e^{-\varepsilon n},\, n\geq 0\rbrace
\end{align*}
and 
\[D^+_{\varepsilon}:=\bigcup_{\ell\geq 1} D^+_{\varepsilon, \ell},\qquad D^-_{\varepsilon}:=\bigcup_{\ell\geq 1} D^-_{\varepsilon, \ell}.\]
It can be shown that $V^u(z)\subseteq D^-_{\varepsilon}$ for any $z\in D^-_{\varepsilon}$, see \cite[Proposition 5]{Pesin92}. Analogously to (\ref{E:localtoglobal}), the global stable and unstable manifolds are defined as
\[W^s(z)=\bigcup_{n\geq0} {\hat{f}}^{-n}(V^s(f^n(z))),\quad z\in D^+_{\varepsilon},\]
and
\[W^u(z)=\bigcup_{n\geq0} {\hat{f}}^{n}(V^u(f^{-n}(z))),\quad z\in D^-_{\varepsilon},\]
where we write ${\hat{f}}^n(A):=f^n(A\setminus N^+)$ and ${\hat{f}}^{-n}(A):=f^{-n}(A\setminus N^-)$,  $A\subseteq \Lambda$.

Now let $\varepsilon>0$ and $\ell\geq 1$ be fixed. Given $x\in D_{\varepsilon,\ell}^-$, we write $B(z,\delta)$ to denote the open ball in $U$ centered at $z$ and with radius $\delta$, and we write $B^u(z,\delta)$ for the open ball in $W^u(z)$ with center $z$ and radius $\delta$. By \cite[Proposition 7]{Pesin92}, there are $r^{(1)}_{\ell}>r^{(2)}_{\ell}>r^{(3)}_{\ell}>0$ such that for any $x\in D_{\varepsilon,\ell}^-$ and any $z\in B(x,r^{(3)}_{\ell})\cap D^-_{\varepsilon, \ell}$ the intersection $V^u(z)\cap W(x)$ of $V^u(z)$ and $W(x):=\exp_x\lbrace v\in E^s(x):\, \Vert v\Vert\leq r^{(1)}_{\ell}\rbrace$ is precisely a single point $[z,x]$ and, moreover, $B^u([z,x],r_\ell^{(2)})\subseteq V^u(z)$. Given $x\in D_{\varepsilon,\ell}^-$ and $\delta\leq r^{(3)}_\ell$ we define the rectangle $\mathcal{R}_{\varepsilon,\ell}(x,\delta)$ by 
\begin{equation}\label{E:rectangleGH}
\mathcal{R}_{\varepsilon,\ell}(x,\delta)=\bigcup_{z\in B(x,\delta)\cap D^-_{\varepsilon, \ell}}B^u([z,x],r^{(2)}_{\ell}).
\end{equation}
Obviously (\ref{E:rectangleGH}) is a partition of $\mathcal{R}_{\varepsilon,\ell}(x,\delta)$ into the sets $B=B^u([z,x],r^{(2)}_{\ell})\subseteq V^u(z)$.

In the following we consider $\varepsilon>0$ to be fixed and therefore suppress it from notation. That is, we write 
\[D^-:=D^-_\varepsilon,\quad D^-_\ell:=D^-_{\varepsilon,\ell},\quad \mathcal{R}_{\ell}(x,\delta):=\mathcal{R}_{\varepsilon,\ell}(x,\delta)\] 
and so on, this shortcut notation follows \cite[p. 129]{Pesin92}.

We call an $f$-invariant Borel probability measure $\mu$ on $D^-$ a \textit{Gibbs u-measure} if for any $\ell\geq 1$, $x\in D_\ell^-$ and $\delta\leq r^{(3)}_\ell$ such that $\mu(\mathcal{R}_{\ell}(x,\delta))>0$ we have 
\begin{equation}\label{E:disintegrationGH}
\mu(E)=\int_{\mathcal{P}_{\mathcal{R}_\ell(x,\delta)}} \, \mu_{B}(E)\, \hat{\mu}_{\mathcal{P}_{\mathcal{R}_\ell(x,\delta)}}(dB), \quad \text{$E\subseteq \mathcal{R}_\ell(x,\delta)\ $ Borel},
\end{equation}
with conditional measures $\mu_B$ equivalent to the Riemannian volumes $m_B$ on the partition elements $B$ as in (\ref{E:rectangleGH}). 

The following existence result for Gibbs u-measures had been shown in \cite[Theorem 1]{Pesin92}.

\begin{theorem}\label{T:GibbsGH}
Let $\Lambda$ be a uniformly hyperbolic attractor with singularities for the $C^2$-diffeomorphism $f$ and assume that there are a point $z\in D^-$ and constants $c>0$, $q>0$, $\varepsilon_0>0$ such that for any $0<\varepsilon\leq \varepsilon_0$, and $n\geq0$
\begin{equation}m_{V^u(z)}(V^u(z)\cap f^{-n}(U(\varepsilon, N^+)))\leq c\:\varepsilon^q,
\end{equation}
where $U(\varepsilon, N^+)$ is the $\varepsilon$-parallel neighborhood of $N^+$ in $M$. Then there is a Gibbs u-measure $\mu\in \mathcal{M}^{ac}_{bd}$ with uniformly bounded densities on $D^-\subseteq \Lambda$, and $\mu(D^-)=1$.
\end{theorem}
\begin{remark}
It follows in particular that $\mu(N^+)=0$.
\end{remark}

Theorem \ref{T:GibbsGH} allows to apply Theorem \ref{T:closable} and its consequences.

\begin{corollary}\label{C:mainresultsGH}
Let $f$ and $\Lambda$ be as in Theorem \ref{T:GibbsGH} and let $\mu\in \mathcal{M}^{ac}_{bd}$ be a Gibbs u-measure with uniformly bounded densities on $D^-$. Then $(\mathcal{E}^{(\mu)},\overline{\mathcal{D}}(\mathcal{E}^{(\mu)}))$ and $(\mathcal{E}^{(\mu)},\mathcal{D}(\mathcal{E}^{(\mu)}))$ as in Theorem \ref{T:closable} are local Dirichlet forms on $L^2(D^-,\mu)$, their generators are self-adjoint operators on $L^2(D^-,\mu)$, and their semigroups are symmetric.
\end{corollary}

\begin{remark}
In general $D^-$ is a proper subset of $\Lambda$, and tangent spaces in the unstable directions are defined only at points of $D^-$. But since $D^-$ has full measure, we have  $L^2(\Lambda,\mu)=L^2(D^-,\mu)$, so that the Dirichlet forms constructed in Corollary \ref{C:mainresultsGH} and the associated Laplacians and semigroups may be regarded as objects on $L^2(\Lambda,\mu)$, and in that sense 'on $\Lambda$'.
\end{remark}

\begin{proof}
The present hypotheses fit into Section \ref{S:Dirichlet} with $D^-$ and $D_\ell^-$, $\ell\geq 1$, as defined here. For each $\ell\geq 1$ finitely many rectangles of type $\mathcal{R}_{\ell}(x,\delta)$ cover the compact set $D_\ell^-\subseteq \Lambda$, and each rectangle admits a partition as in (\ref{E:rectangleGH}). The measure $\mu$ satisfies the disintegration identities (\ref{E:disintegration}) in the form (\ref{E:disintegrationGH}) and with uniformly bounded conditional densities. 
\end{proof}

We provide some examples for attractors with singularities.

\begin{examples}
Let $I=(-1,1)$, $U=I\times I$ and $N=I\times\lbrace 0\rbrace$ and let $f:U\setminus N\rightarrow U$ be a map of the form $f(x,y):=(g(x,y),h(x,y))$, where $g,h$ are functions given by
\begin{align*}
g(x,y)&=(-B\vert y\vert^{\nu_0}+Bx\,\text{sgn}\,y\vert y\vert^\nu+1)\,\text{sgn}\,y,\\
h(x,y)&=((1+A)\vert y\vert^{\nu_0}-A)\,\text{sgn}\,y,
\end{align*}
for constants $0<A<1$, $0<B<\frac{1}{2}$, $\nu>1$, $1/(1+A)<\nu_0<1$. The resulting attractor is the well-known \emph{(geometric) Lorenz attractor}, it is a uniformly hyperbolic attractor with singularities. 

A more common definition of the Lorenz attractor is as the attractor (in ODE sense) for the non-linear system
\[\dot{x}=-\sigma x+\sigma y, \quad  \dot{y}=rx-y-xz \quad \text{and}\quad  \dot{z}=xy-bz \]
for the particular parameters $\sigma=10$, $b=\frac{8}{3}$ and $r=28$, illustrated in Figure 1. 

Further details and more general classes of Lorenz attractors are discussed in \cite[Section 5.2]{Pesin92}, see also \cite[Section 13.3]{HaPe03} and \cite[Section 2.2]{Kuznetsov}.

\begin{figure}
\centering
\includegraphics[height=5cm]{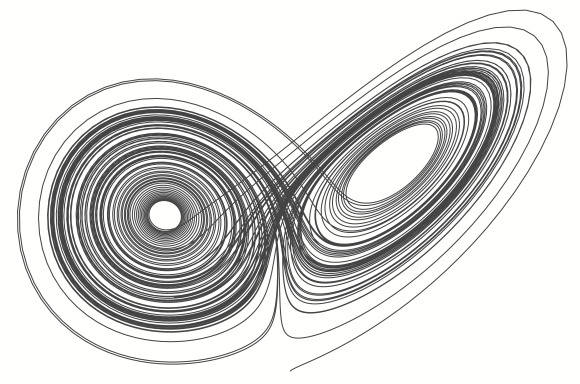}
\caption{Lorenz attractor}
\end{figure}
\end{examples}

\begin{examples}
Given $c\in (0,1)$,  let $I=(-c,c)$, $U=I\times I$ and $N=\lbrace0\rbrace\times I$. Define $f:U\setminus N\rightarrow U$ by $f(x,y)=(1+by-a\vert x\vert,x)$, where $0<a<a_0$ and $0<b<b_0$ for some small $a_0, b_0>0$. The map $f$ is called the \emph{Lozi map}, and the associated attractor is the \emph{Lozi attractor}, \cite{Lozi}. This is the special case of the so-called Lozi-like maps studied in\cite{Young, Pesin92}. Ergodic and topological properties can be found in \cite{Levy, Misiurewicz}.
\end{examples}

\begin{examples}
Let $U=(-1,1)\times (-1,1)$ and $N=\lbrace(x,y): y=kx\rbrace\subset U$. The map $f:U\setminus N\rightarrow U$ given by
\begin{equation*}
  f(x,y)=  \begin{cases}
      (\lambda_1(x-1)+1,\,\lambda_2(y-1)+1)\quad \text{for}\,\, y>kx, & \\
      (\mu_1(x+1)-1,\,\mu_2(y+1)-1)\quad \text{for}\,\, y<kx, & \\
    \end{cases}      
\end{equation*}
has a hyperbolic attractor with singularity set $N$ whenever $\vert k\vert<1$, $0<\lambda_1, \mu_1<\frac{1}{2}$ and $1<\lambda_2,\mu_2<\frac{2}{1-\vert k\vert}$. It is called the \emph{Belykh attractor}. Details can be found in \cite[p. 149]{Pesin92} and \cite{Sataev99}.
\end{examples}

\begin{remark}
Partially hyperbolic attractors with singularities $\Lambda=\overline{D}$ can similarly be constructed using (\ref{attractor}), with the difference that $D$ must be a partially hyperbolic set in the sense of Section \ref{S:PH}. Stable and unstable manifolds can be constructed analogously, as well as Gibbs u-measures, \cite[Theorem 12]{Pesin92}. A geometric example of such attractors is provided in \cite[Section 5.2, Example 4]{Pesin92}. See also \cite{LS82}.
\end{remark}

%

\end{document}